\documentclass[a4paper,11pt]{amsart}

\title[Optimal embeddings into Lorentz spaces]{Optimal embeddings into Lorentz spaces for some vector differential operators via Gagliardo's lemma}

\author{Daniel Spector}
\address{
National Chiao Tung University \\
Department of Applied Mathematics\\
Hsinchu, Taiwan}
\address{
National Center for Theoretical Sciences\\ 
National Taiwan University\\
No. 1 Sec. 4 Roosevelt Rd.\\
Taipei, 106, Taiwan}
\address{
Washington University in St. Louis\\ 
Department of Mathematics\\
One Brookings Drive\\
St. Louis, MO 63130-4899}
\email{dspector@math.nctu.edu.tw}

\author{Jean Van Schaftingen}
\address{Universit\'e catholique de Louvain\\ 
Institut de Recherche en Math\'ematique et Physique\\
Chemin du Cyclotron 2 bte L7.01.01\\
1348 Louvain-la-Neuve\\
Belgium}
\email{Jean.VanSchaftingen@uclouvain.be}

\usepackage[T1]{fontenc}
\usepackage[utf8]{inputenc}
\usepackage{geometry}
\usepackage{microtype}
\usepackage{lmodern} 
\usepackage{stmaryrd}

\usepackage{amssymb}
\usepackage{amsmath}
\usepackage{amsthm}
\usepackage{esint}

\usepackage[foot]{amsaddr}
\usepackage{constants}
\usepackage{paralist}
\usepackage{mathrsfs}
\newcommand{\defeq}{\triangleq}
\newcommand{\compose}{\,\circ\,}

\usepackage{todonotes}

\usepackage[pdftitle={Optimal embeddings into Lorentz spaces for some vector differential operators via Gagliardo's lemma},%
pdfauthor={Daniel Spector and Jean Van Schaftingen},%
final]{hyperref}

\usepackage[nameinlink%
,capitalize%
]{cleveref}

\usepackage[abbrev,backrefs]{amsrefs}

\newtheorem{theorem}{Theorem}
\newtheorem{proposition}{Proposition}[section]
\newtheorem{lemma}[proposition]{Lemma}

\newtheorem{question}[proposition]{Question}

\theoremstyle{definition}
\newtheorem{definition}[proposition]{Definition}

\theoremstyle{remark}

\numberwithin{equation}{section}

\newcommand{\abs}[1]{{\lvert #1 \rvert}}
\newcommand{\bigabs}[1]{{\bigl\lvert #1 \bigr\rvert}}

\newcommand{\biggabs}[1]{{\biggl\lvert #1 \biggr\rvert}}
\newcommand{\norm}[2][]{{\lVert #2 \rVert}_{#1}}

\newcommand{\dualprod}[2]{\langle #1, #2 \rangle}

\newcommand{\st}{\;:\;}

\newcommand{\Rset}{\mathbb{R}}
\newcommand{\Nset}{\mathbb{N}}

\newcommand{\dif}{\,\mathrm{d}}

\DeclareMathOperator{\linspan}{span}
\allowdisplaybreaks
\begin{document}

\begin{abstract}
%arXiv compatible LaTeX in the abstract!
We prove a family of Sobolev inequalities of the form 
$$
  \Vert u \Vert_{L^{\frac{n}{n-1}, 1} (\mathbb{R}^n,V)} 
  \le 
  C
  \Vert A (D) u \Vert_{L^1 (\mathbb{R}^n,E)} 
$$
where $A (D) : C^\infty_c (\mathbb{R}^n, V) \to C^\infty_c (\mathbb{R}^n, E)$ is a vector first-order homogeneous linear differential operator with constant coefficients, $u$ is a vector field on $\mathbb{R}^n$ and $L^{\frac{n}{n - 1}, 1} (\mathbb{R}^{n})$ is a Lorentz space. These new inequalities imply in particular the extension of the classical Gagliardo--Nirenberg inequality to Lorentz spaces originally due to Alvino and a sharpening of an inequality in terms of the deformation operator by Strauss (Korn--Sobolev inequality) on the Lorentz scale. The proof relies on a nonorthogonal application of the Loomis--Whitney inequality and Gagliardo's lemma.
\end{abstract}

\maketitle
\section{Introduction and Main Results}
A now classical result of Gagliardo \cite{Gagliardo} and Nirenberg \cite{Nirenberg} asserts the existence of a constant $C>0$ such that the inequality
\begin{align}\label{GN}
\|u \|_{L^{n/(n-1)}(\Rset^n)} \leq C \|D u \|_{L^1(\Rset^n,\Rset^n)}
\end{align}
holds for all $u \in W^{1,1}(\Rset^n)$.  While optimal on the scale of Lebesgue spaces, one can improve the target to a better Lorentz space.  
Indeed, Alvino \cite{Alvino} proved that there exists a constant \(C'\) such that the inequality
\begin{align}%
\label{Alvino}
\|u \|_{L^{n/(n-1),1}(\Rset^n)} \leq C' \|D u \|_{L^1(\Rset^n,\Rset^n)}
\end{align}
holds for all functions $u \in W^{1,1}(\Rset^n)$ (see \cref{preliminaries} for a precise definition of the Lorentz space $L^{p,q}(\Rset^n)$), with an explicit optimal value of the constant \(C'\).
The estimate \eqref{Alvino} reaches a limiting case of the class of Sobolev embeddings into Lorentz spaces treated by O'Neil and Peetre \citelist{\cite{ONeil_1963}*{\S 3}\cite{Peetre_1966}*{Théorème 7.1}} for $u \in W^{1,p}(\Rset^n)$, $p>1$; the corresponding improved Lorentz estimates when \(p = n\) lead to exponential integrability estimates \citelist{\cite{Brezis_Wainger_1980}*{Theorems 2 and 3}\cite{Brezis_1979}*{Theorem 7}}.
The inequality \eqref{Alvino} was rediscovered by Poornima \cite{Poornima} and Tartar \cite{Tartar_1998}*{Theorem 8}, and was also proved by Fournier \cite{Fournier}.
As $L^{n/(n-1),1}(\Rset^n) \subsetneq L^{n/(n-1), n/(n-1)}(\Rset^n)=L^{n/(n-1)}(\Rset^n)$, the inequality \eqref{Alvino} improves \eqref{GN}, while simple examples show that one cannot obtain further improvement in the second parameter.  That one should be interested in the Lorentz spaces in general, or the sharpening of the inequality \eqref{GN} found in \eqref{Alvino} in particular, can be seen from a number of perspectives.
A first motivation comes from real interpolation of Banach spaces \citelist{\cite{Lions_Peetre_1961}\cite{Lions_Peetre_1964}\cite{Aronszajn_Gagliardo_1965}\cite{Lions_1959}\cite{Gagliardo_1960}}, in which the Lorentz spaces arise readily
\begin{equation*}
  L^{p,q}(\Rset^n) 
  = 
    \bigl(
      L^1(\Rset^n),L^\infty(\Rset^n)
    \bigr)_{1 - \frac{1}{p},q}
    \,,
\end{equation*}
(see, e.g. \citelist{\cite{Triebel_1978}*{Th\'eor\`eme 1.18.16.1}\cite{Lions_Peetre_1964}*{Remarque (2.1)}\cite{Tartar_2007}*{Lemma 22.6}\cite{Adams_Fournier_2003}*{Theorem 7.26}}),
while the weak-$L^p$ space $L^{p,\infty}(\Rset^n)$ is a natural space for many harmonic analysis estimates, where weak-type endpoints can be upgraded to strong-type interpolated estimates.  In our specific considerations, the improvement in the second parameter is more than microscopic, as it encodes significantly more information in the trade off between differentiability and integrability than the classical inequality \eqref{GN}.  One perspective of this gain is that from \eqref{Alvino} it is possible to deduce Hardy's inequality 
\begin{align}\label{hardyineq}
    \int_{\Rset^n} 
      \frac
        {\abs{u (x)}}
        {|x|}
      \dif x 
  \leq 
    \tilde{C} 
    \int_{\Rset^n} 
      \abs{Du (x)}
      \dif x,
\end{align}
by a simple application of H\"older's inequality on the Lorentz space scale.

A vector analogue of \eqref{GN} follows easily from the same argument, yet such an inequality is not optimal, as one does not need the full gradient in order to obtain an embedding into $L^{n/(n-1)}(\Rset^n)$.  For example, a result of M.J. Strauss \cite{Strauss} shows that if one defines the symmetric part of the gradient
\begin{align*}
    Eu
  \defeq 
    \tfrac{1}{2}\bigl(D u + (D u)^T)
\end{align*}
(\(Eu\) is known in  elasticity as the linearized deformation tensor associated to the displacement \(u\)),
then one has the existence of a constant $C''>0$ such that
\begin{align}\label{Strauss}
\|u \|_{L^{n/(n-1)}(\Rset^n)} \leq C'' \| Eu \|_{L^1(\Rset^n,\Rset^n)}
\end{align}
for all vector fields $u \in W^{1,1}(\Rset^n,\Rset^n)$.  
The need for such inequalities arose in the work of Duvaut and Lions \cite{Duvaut_Lions}, while interest in the study of such spaces has expanded greatly into the theory of functions of bounded deformation \cite{Temam-Strang, Babadijan, BFT, DalMaso}.  The inequality \eqref{Strauss} is a special application of Strauss's work, which can be deduced from a more refined inequality \cite{Strauss}*{p.\thinspace{}208} in the spirit of a preceding work in the $L^2$ case due to De Figueiredo \cite{DeFigueiredo}. 

Similar estimates have been proved for other differential operators, including the Hodge complex, in a series of works initiated by Bourgain and Brezis (see \citelist{\cite{Bourgain_Brezis_2004}\cite{Bourgain_Brezis_2007}\cite{VanSchaftingen_2004}\cite{Lanzani_Stein_2005}}), while more generally, the second author has shown that the vector differential inequality
\begin{equation}
\label{ineq_EC}
    \norm[L^{\frac{n}{n - 1}}(\Rset^n,V)]{u} 
  \le 
    \tilde{\tilde{C}}
    \norm[L^1(\Rset^n, E)]{A (D) u}
\end{equation}
holds
for every vector field \(u \in C^\infty_c (\Rset^n, V)\)
if and only if the homogeneous first-order linear vector differential operator with constant coefficients \(A (D)\) is elliptic and canceling \cite{VanSchaftingen_2013}*{Theorem 1.3} (see also \cref{canceling}).

Given the sharpening of the inequality \eqref{GN} obtained in \eqref{Alvino}, a natural question is whether all elliptic and canceling operators admit such improvements (see \citelist{ \cite{VanSchaftingen_2013}*{Open problem 8.3}\cite{VanSchaftingen_2015}*{Open problem 2}\cite{Bourgain_Brezis_2007}*{Open problem 1}
\cite{VanSchaftingen_2010}*{Open problem 2}}). 
By relying on estimates into fractional Sobolev spaces and embeddings of these into Lorentz spaces, it has been proved that such improvements can be obtained in the Lorentz space \(L^{\frac{n}{n - 1}, q} (\Rset^n)\) for every \(q > 1\) %
\citelist{%
\cite{VanSchaftingen_2010}*{Theorem 3}%
\cite{VanSchaftingen_2013}*{Theorem 8.5}%
}.
The possibility of extending \eqref{Alvino} to elliptic and canceling operators is supported by the fact that the analogue of the  Hardy inequality \eqref{hardyineq} for such operators that it would imply is known to hold \cite{Bousquet_VanSchaftigen_2014}.

\medbreak

The main result of this paper is the positive answer to such a result in the plane, and a partial answer in higher dimensions, that such an embedding holds for elliptic and $(n-1)$-canceling operators.  From this we show how one can deduce the inequality \eqref{Alvino}, as well as the following sharpening of Strauss' result \eqref{Strauss}.

\begin{theorem}
\label{proposition_Korn_Sobolev}
There exists a constant $C>0$ such that
\[
    \norm[L^{\frac{n}{n - 1},1}(\Rset^n,\Rset^n)]{u}
  \le 
    C
    \norm[L^1(\Rset^n,\Rset^{n\times n})]{E u}
\]
for every \(u \in C^\infty_c (\Rset^n,\Rset^n)\).
\end{theorem}

Luc Tartar mentioned in 2012 to the second author that he had a proof of \cref{proposition_Korn_Sobolev} that has not yet been published; it appeared afterwards that both our independent proofs of \cref{proposition_Korn_Sobolev} were following the same strategy.

\medbreak 

The idea underlying the improvement of \cref{proposition_Korn_Sobolev} is contained in the following general theorem which allows one to control the Lorentz norm by a product of directional derivatives.

\begin{theorem}
\label{proposition_directional_derivatives}
Let \(\ell \in \Nset\), let, for \(i \in \{1, \dotsc, n\}\) and \(j \in \{1, \dotsc, \ell\}\),
\(w_i^j \in \Rset^n\) and \(v_i^j \in V^*\).
If for every \(j \in \{1, \dotsc, \ell\}\), the vectors \(w_1^j, \dotsc, w_n^j\) are linearly independent in \(\Rset^n\)
and if 
\[
  \bigcap_{j = 1}^\ell \bigcup_{i = 1}^n (v_i^j)^\perp = \{0\},
\]
then for every function \(u \in C^\infty_c (\Rset^n, V)\),
\begin{equation*}
    \norm[L^{\frac{n}{n - 1}, 1}(\Rset^n,V)]{u}
  \le
    C
    \sum_{j=1}^\ell 
      \norm[L^1(\Rset^n)]{\dualprod{v_1^j}{D u[w_1^j]}}^\frac{1}{n}
      \dotsm
      \norm[L^1(\Rset^n)]{\dualprod{v_n^j}{D u[w_n^j]}}^\frac{1}{n}
  .    
\end{equation*}
\end{theorem}

From \cref{proposition_directional_derivatives}, we obtain a complete answer in the two-dimensional case:

\begin{theorem}
\label{theorem_canceling_2d}
Let $V$ and $E$ be finite-dimensional spaces and let \(A (D):C^\infty_c (\Rset^2, V) \to C^\infty_c (\Rset^2, E)\) be a first-order homogeneous linear differential operator with constant coefficients. 
There exists a constant $C>0$ such that
\[
    \norm[L^{2, 1}(\Rset^2,V)]{u} 
  \le 
    C
    \norm[L^1(\Rset^2,E)]{A (D) u}
\]
for every \(u \in C^\infty_c (\Rset^2, V)\) if and only if the operator \(A (D)\) is elliptic and canceling.
\end{theorem}

\Cref{theorem_canceling_2d} shows in particular that when \(n = 2\) the Sobolev inequality \eqref{ineq_EC} holds if and only if the corresponding limiting estimate in Lorentz spaces holds.

The proof of \cref{theorem_canceling_2d} in \cref{canceling} shows in higher dimensions \(n > 2\) that one has such inequalities for all elliptic and $(n-1)$--canceling operators. 

\Cref{theorem_canceling_2d} motivates the following

\begin{question}
Let \(n \ge 3\) and let $V$ and $E$ be finite-dimensional spaces.  Further suppose that the first-order homogeneous linear differential operator  with constant coefficients \(A (D):C^\infty_c (\Rset^n, V) \to C^\infty_c (\Rset^n, E)\) is elliptic and canceling.  Can one show the existence of a constant $C>0$ such that the inequality
\[
    \norm[L^{\frac{n}{n - 1},1}(\Rset^n,V)]{u}
  \le 
    C
    \norm[L^1(\Rset^n,E)]{A (D) u}
\]
holds for every \(u \in C^\infty_c (\Rset^n, V)\)?
\end{question}

Another open question stems from the fact that higher-order operators satisfy Sobolev estimates \cite{VanSchaftingen_2013}*{Theorem 1.3}, Hardy inequalities \cite{Bousquet_VanSchaftigen_2014}*{Theorem 1.2} and non-optimal Lorentz estimates \cite{VanSchaftingen_2013}*{Theorem 8.5}.

\begin{question}
Let \(n \ge 2\), \(k \ge 2\) and let $V$ and $E$ be finite-dimensional spaces.  Further suppose that the \(k\)--th order homogeneous linear differential operator  with constant coefficients \(A (D):C^\infty_c (\Rset^n, V) \to C^\infty_c (\Rset^n, E)\) is elliptic and canceling.  Can one show the existence of a constant $C>0$ such that the inequality
\[
    \norm[L^{\frac{n}{n - 1},1}(\Rset^n,V)]{D^{k - 1} u}
  \le 
    C
    \norm[L^1(\Rset^n,E)]{A (D) u}
\]
holds for every \(u \in C^\infty_c (\Rset^n, V)\)?
\end{question}

The plan of the paper is as follows.  In \cref{preliminaries}, we recall some requisite preliminaries regarding Lorentz spaces.  In \cref{plane}, we give a proof of  \cref{proposition_Korn_Sobolev} in the plane, as it illustrates well the idea of the more general \cref{proposition_directional_derivatives}.  
The main ingredient for higher dimensions is a version of an inequality of Loomis and Whitney \cite{Loomis_Whitney_1949} to nonorthogonal coordinate systems. We show how this can  be obtained from a change of coordinates and a Lemma of Gagliardo in \cite{Gagliardo} in a presentation that tries to keep the geometric content of the inequality.  
In \cref{proofs} we prove \cref{proposition_directional_derivatives}, from which \cref{proposition_Korn_Sobolev} is deduced in \cref{section_sirect_consequences}.  Finally, in \cref{canceling} we prove a general result for elliptic and $(n-1)$--canceling operators.

\section{Preliminaries}\label{preliminaries}
In the sequel, with an abuse of notation we utilize  $|\cdot|$ to denote the norm in any finite-dimensional vector space, e.g. the absolute value, the norm in Euclidean space $\Rset^n$, and the norm in $V$ and $E$. 

For \(n \in \Nset\), $1<p<+\infty$ and $1 \leq q \leq +\infty$, we denote by $L^{p,q} (\Rset^n)$ the Lorentz space \cite{Lorentz_1950} (see also for example \cite{Grafakos}) with quasinorm
\begin{equation}
\label{def_Lorentz_norm}
    \norm[L^{p,q}(\Rset^n)]{u}^q 
  \defeq 
    p 
    \int_0^{+\infty} 
    \bigl(t \, \mathcal{L}^n(\{ x \in \Rset^n \st \abs{u (x)}>t\})^{1/p}\bigr)^q 
    \,
    \frac{\dif t}{t},
\end{equation}
where \(\mathcal{L}^n (A)\) denotes the Lebesgue measure of a measurable set \(A \subset \Rset^n\).
Equivalently, if the function \(\abs{u}^* : (0, +\infty) \to [0, + \infty)\) is the nonincreasing rearrangement of \(\abs{u}\), that is, if for every \(t \in (0, + \infty)\), \(\mathcal{L}^1 (\{ s \in [0, + \infty) \st \abs{u}^* (s) > t\})
= \mathcal{L}^n (\{ x \in \Rset^n \st \abs{u (x)} > t\})\),
then 
\begin{equation*}
    \|u \|^q_{L^{p,q}(\Rset^n)}
 =
    p
    \int_0^{+\infty} 
    \bigl(t\, \mathcal{L}^1 (\{ s \in (0, +\infty) \st \abs{u}^* (s)>t\})^{1/p}\bigr)^q \frac{\dif t}{t}
 =
    \int_0^{+ \infty} 
      \abs{u}^* (s)\, s^\frac{q}{p}
      \,
      \frac{\dif s}{s}\,.
\end{equation*}
In particular, Cavalieri's principle shows that $L^{p,p}(\Rset^n)=L^p(\Rset^n)$, while the spaces are nested increasingly with respect to the second parameter:
\begin{align*}
L^{p,1}(\Rset^n) \subset L^{p,q}(\Rset^n) \subset L^{p,\infty}(\Rset^n).
\end{align*}
The quantity \(\norm[L^{p, q} (\Rset^n)]{\cdot}\) is a norm when \(q \le p\) \cite{Lorentz_1950}*{Theorem 1}.

The definition of Lorentz spaces and norms by \eqref{def_Lorentz_norm} is equivalent when \(p > 1\) and \(q \ge 1\) to the definition by interpolation (see for example \citelist{\cite{Triebel_1978}*{Th\'eor\`eme 1.18.16.1}\cite{Lions_Peetre_1964}*{Remarque (2.1)}\cite{Tartar_2007}*{Lemma 22.6}\cite{Adams_Fournier_2003}*{Theorem 7.26}}) and to the definition through \emph{averaged rearrangements} which is also common \cite{Ziemer_1989}*{\S 1.8}.

\section{Proof of \cref{proposition_Korn_Sobolev} in the Planar case}\label{plane}
We here give a proof of \cref{proposition_Korn_Sobolev} in the plane \(\Rset^2\). That is, we suppose that the function $u:\Rset^2 \to \Rset^2$ is smooth and has compact support, and we will obtain an estimate for
\begin{align*}
    \|u_1\|_{L^{2,1}(\Rset^2)} 
  = 
    2
    \int_0^\infty 
      \mathcal{L}^2 \bigl(\{ x \in \Rset^2 \st |u_1 (x) |>t\}\bigr)^\frac{1}{2} 
      \dif t
\end{align*}
by the quantity
\begin{align*}
&\int_{\Rset^2} \left| \frac{\partial u_1}{\partial x_1}\right|+ \left| \frac{\partial u_1}{\partial x_2} + \frac{\partial u_2}{\partial x_1}\right|+ \left| \frac{\partial u_2}{\partial x_2}\right|, \\
\end{align*}
as an analogous argument implies a similar inequality for $u_2$.  

First let us recall Fournier's argument \cite{Fournier}*{Appendix} of how to use the Loomis--Whitney inequality to obtain the embedding \eqref{Alvino} where one assumes the full derivative \(Du\) is in $L^1$.  In our setting of the plane this reduces to the degenerate case of the Loomis--Whitney inequality that the area \(A\) of a set can be bounded by the product of its length \(l\) and width \(w\):
\begin{align}\label{areainequality}
A \leq l\times w.
\end{align}
More specific to our problem, this takes for every \(t > 0\) the form of the inequality
\begin{multline*}
\mathcal{L}^2 (\{|u_1|>t\}) \\
  \leq 
    \mathcal{H}^1\Bigl(\Bigl\{x_1 \in \Rset \st \sup_{x_2 \in \Rset} |u_1(x_1,x_2)|>t\Bigr\}\Bigr)\, 
    \mathcal{H}^1\Bigl(\Bigl\{x_2 \in \Rset \st \sup_{x_1 \in \Rset} |u_1(x_1,x_2)|>t\Bigr\}\Bigr),
\end{multline*}
where \(\mathcal{H}^1\) denotes the one-dimensional Hausdorff measure and 
\(\{|u_1|>t\} = \{(x_1,x_2) \in \Rset^2 : |u_1(x_1,x_2)|>t\}\).
One then observes that a separate integration of the two terms on the right-hand side yields
\begin{gather*}
\int_0^\infty \mathcal{H}^1\Bigl(\Bigl\{x_1 \in \Rset \st \sup_{x_2 \in \Rset} |u_1(x_1,x_2)|>t\Bigr\}\Bigr)\dif t 
= \int_{\Rset} \sup_{x_2 \in \Rset} |u_1(x_1,x_2)|\dif x_1 
\\
\intertext{and}
\int_0^\infty \mathcal{H}^1\Bigl(\Bigl\{x_2 \in \Rset \st \sup_{x_1 \in \Rset} |u_1(x_1,x_2)|>t\Bigr\}\Bigr)\dif t = \int_{\Rset} \sup_{x_1 \in \Rset} |u_1(x_1,x_2)|\dif x_2,
\end{gather*}
which can be further estimated by the derivatives via the fundamental theorem of calculus as
\begin{gather}
\label{eq_KwEMldOpUbn}
\int_{\Rset} \sup_{x_2 \in \Rset} |u_1(x_1,x_2)|\dif x_1  \leq \int_{\Rset^2} \left|\frac{\partial u_1}{\partial x_2}(x_1,x_2)\right|\dif x_2 \dif x_1,\\
\intertext{and}
\label{eq_wcnn0jIC0UC}
\int_{\Rset} \sup_{x_1 \in \Rset} |u_1(x_1,x_2)|\dif x_2 \leq \int_{\Rset^2} \left|\frac{\partial u_1}{\partial x_1}(x_1,x_2)\right|\dif x_1 \dif x_2.
\end{gather}
Therefore, from the definition of the Lorentz norm \eqref{def_Lorentz_norm} one has 
\begin{multline*}
\int_0^\infty \mathcal{L}^2 \bigl(\{ |u_1|>t\}\bigr)^\frac{1}{2} \dif t
\leq \int_0^\infty \biggl(\mathcal{H}^1\Bigl(\Bigl\{x_1 \in \Rset \st \sup_{x_2 \in \Rset} |u_1(x_1,x_2)|>t\Bigr\}\Bigr)\\
\times  \mathcal{H}^1\Bigl(\Bigl\{x_2 \in \Rset \st \sup_{x_1 \in \Rset} |u_1(x_1,x_2)|>t\Bigr\}\Bigr)\biggr)^\frac{1}{2} \dif t,
\end{multline*}
while the Cauchy--Schwarz inequality implies
\begin{multline}
\label{eq_XA4vzjlx0LX}
\int_0^\infty \mathcal{L}^2 \bigl(\{ |u_1|>t\}\bigr)^\frac{1}{2} \dif t 
\leq \left(\int_0^\infty \mathcal{H}^1\Bigl(\Bigl\{x_1 \in \Rset \st \sup_{x_2 \in \Rset} |u_1(x_1,x_2)|>t\Bigr\}\Bigr)\dif t\right)^\frac{1}{2} 
\\
\times \left(\int_0^\infty \mathcal{H}^1\Bigl(\Bigl\{x_2 \in \Rset \st \sup_{x_1 \in \Rset} |u_1(x_1,x_2)|>t\Bigr\}\Bigr)\dif t \right)^\frac{1}{2}.
\end{multline}
Finally one combines the inequalities \eqref{eq_KwEMldOpUbn}, \eqref{eq_wcnn0jIC0UC} and \eqref{eq_XA4vzjlx0LX} to deduce the multiplicative inequality
\begin{align*}
\int_0^\infty \mathcal{L}^2 \bigl(\{ |u_1|>t\}\bigr)^\frac{1}{2} \dif t \leq \left(\int_{\Rset^2} \left|\frac{\partial u_1}{\partial x_2}(x_1,x_2)\right|\dif x_2 \dif x_1 \times  \int_{\Rset^2} \left|\frac{\partial u_1}{\partial x_1}(x_1,x_2)\right|\dif x_1 \dif x_2\right)^\frac{1}{2},
\end{align*}
while the arithmetic geometric mean inequality yields the additive form
\begin{align*}
\int_0^\infty \mathcal{L}^2 \bigl(\{ |u_1|>t\}\bigr)^\frac{1}{2} \dif t \leq \frac{1}{2} \int_{\Rset^2} \left|\frac{\partial u_1}{\partial x_2}(x_1,x_2)\right|\dif x_2 \dif x_1+ \frac{1}{2} \int_{\Rset^2} \left|\frac{\partial u_1}{\partial x_1}(x_1,x_2)\right|\dif x_1 \dif x_2.
\end{align*}

Now, this argument is not sufficient to obtain the Korn--Sobolev inequality, since in general one has no control over the quantity
\begin{align*}
\int_{\Rset^2} \left|\frac{\partial u_1}{\partial x_2}(x_1,x_2)\right|\dif x_2 \dif x_1.
\end{align*}
However, in this setting one still assumes the finiteness of 
\begin{align*}
\int_{\Rset^2} \left|\frac{\partial u_1}{\partial x_1}(x_1,x_2)\right|\dif x_2 \dif x_1,
\end{align*}
which tracing back through the inequalites translates to control over the width $w$ in \eqref{areainequality}.  In general we cannot hope to control the length $l$ in this way, but it turns out we can control measurements in certain other directions.  In particular, we can estimate the measurement of length in both directions whose angle with the $x_2$ axis is $\pi/4$.  In either case the measurement of $l'$ gives us an upper bound on an estimate for $l$ by simple trigonometry, leading to the inequality
\begin{align}\label{areainequalityprime}
A \leq \sqrt{2} \;l' \times w,
\end{align}
which as we will see will be sufficient to obtain our result.

We now commence with the
\begin{proof}[Proof of \cref{proposition_Korn_Sobolev} in the planar case]
Let us now see how this ability to control the area with respect to nonorthogonal measurements yields the desired inequality.  
First, we note that since, by the triangle inequality,
\begin{equation}
\abs{u_1} \le \frac{\abs{u_1 + u_2} + \abs{u_1 - u_2}}{2},
\end{equation}
on \(\Rset^2\),
we have for every \(t > 0\)
\begin{equation*}
  \{|u_1|>t\} \subset \{|u_1|>t, |u_1+u_2|>t\}\cup  \{ |u_1|>t,|u_1-u_2|>t\},
\end{equation*}
and therefore subadditivity of the measure \(\mathcal{L}^2\) and of the square root implies that 
\begin{equation}
\label{twopieces}
    \mathcal{L}^2 \bigl(\{ |u_1|>t\}\bigr)^\frac{1}{2} 
  \leq  
      \mathcal{L}^2 \bigl(\{ |u_1|>t,|u_1+u_2|>t\}\bigr)^\frac{1}{2}
    +
      \mathcal{L}^2 \bigl(\{|u_1|>t, |u_1-u_2|>t\}\bigr)^\frac{1}{2}.
\end{equation}
Let us estimate the first term on the right-hand side.  We apply the inequality \eqref{areainequalityprime} to deduce that for each \(t > 0\),
\begin{multline*}
  \mathcal{L}^2 \bigl(\{ |u_1|>t,|u_1+u_2|>t\}\bigr)\\
\leq 
  {\sqrt{2}}
  \;
  \mathcal{H}^1
    \Bigl(
      \Bigl\{s \in \mathbb{R} \st \sup_{x_1-x_2=s} |u_1|>t,\sup_{x_1-x_2=s}|u_1+u_2|>t\Bigr\}\Bigr)  \\
  \times 
  \mathcal{H}^1 \Bigl(\Bigl\{ x_1 \in \Rset \st \sup_{x_2 \in \Rset} |u_1|>t, \sup_{x_2 \in \Rset}|u_1+u_2|>t\Bigr\}\Bigr).
\end{multline*}
Then the removal of certain inequalities in the sets only increases the measure, we find for each \(t >0\)
\begin{multline*}
 \mathcal{L}^2(\{ |u_1|>t,|u_1+u_2|>t\})
  \leq 
    {\sqrt{2}}
    \;
    \mathcal{H}^1
      \Bigl(\Bigr\{s \in \mathbb{R} \st \sup_{x_1-x_2=s}|(u_1+u_2) (x_1, x_2)|>t\Bigr\}\Bigr)  \\
\times \mathcal{H}^1\Bigl(\Bigl\{ x_1 \in \Rset: \sup_{x_2 \in \Rset} |u_1 (x_1, x_2)|>t,\Bigr\}\Bigr).
\end{multline*} 
Now while the integral in $t$ of the second term on the right has been computed, for the first we find
\begin{multline*}
    \int_0^\infty 
      \mathcal{H}^1
        \Bigl(\Bigr\{ s \in \Rset \st \sup_{x_1 - x_2 = s} \abs{(u_1 + u_2)(x_1, x_2)}>t \Bigr\}\Bigr) \dif t\\[-.5em]
        = \int_{\Rset} \sup_{x_1-x_2=s}|(u_1+u_2)(x_1,x_2)| \dif s.
\end{multline*}
We claim that this diagonal length can be controlled by the symmetric part of the gradient via the estimate
\begin{align}\label{diagonalclaim}
\int_{\Rset} \sup_{x_1-x_2=s}|(u_1+u_2)(x_1,x_2)| \dif s &\leq \int_{\Rset^2}  \left| \left(\frac{\partial u_1}{\partial x_1} + \frac{\partial u_2}{\partial x_2} +\frac{\partial u_1}{\partial x_2}  + \frac{\partial u_2}{\partial x_1}\right)(x_1,x_2) \right|\dif x_1 \dif x_2,
\end{align}
from which the desired bound can be deduced, as the Cauchy--Schwarz  inequality yields
\begin{multline*}
\int_0^\infty \mathcal{L}^2 (\{ |u_1|>t, |u_1+u_2|>t\})^\frac{1}{2} \dif t\\[-.7em]
\leq \left(\int_0^\infty \mathcal{H}^1\Bigl(\Bigl\{x_1 \in \Rset \st \sup_{x_2 \in \Rset} |u_1(x_1,x_2)|>t\Bigr\}\Bigr)\dif t\right)^\frac{1}{2}\\ 
\times \left(\int_0^\infty \mathcal{H}^1\Bigl(\Bigl\{s \in \mathbb{R}: \sup_{x_1 - x_2=s}|(u_1+u_2)(x_1, x_2)|>t\Bigr\}\Bigr)\dif t \right)^\frac{1}{2},
\end{multline*}
and therefore 
\begin{equation*}
 \begin{split}
\int_0^\infty \mathcal{L}^2 (\{ |u_1|>t&,|u_1+u_2|>t\})^\frac{1}{2} \dif t\\[-.7em]
&\leq \left(\int_{\Rset} \sup_{x_1 \in \Rset} |u_1(x_1,x_2)|\dif x_1\right)^\frac{1}{2} \times\left(\int_{\Rset} \sup_{x_1-x_2=s}|(u_1+u_2)(x_1,x_2)| \dif s\right)^\frac{1}{2} \\
&\leq\left( \int_{\Rset^2} \left|\frac{\partial u_1}{\partial x_1}(x_1,x_2)\right|\dif x_2 \dif x_1\right)^\frac{1}{2} \\
&\qquad\qquad \times \left( \int_{\Rset^2}  \left| \left(\frac{\partial u_1}{\partial x_1} + \frac{\partial u_2}{\partial x_2} +\frac{\partial u_1}{\partial x_2}  + \frac{\partial u_2}{\partial x_1}\right)(x_1,x_2) \right|\dif x_1 \dif x_2\right)^\frac{1}{2}.
 \end{split}
\end{equation*}

It therefore remains to prove the claim \eqref{diagonalclaim}, as well as a similar estimate relating to a bound for the measure of the set $\{ |u_1|>t,|u_1-u_2|>t\}$.  These two estimates are achieved by a modification of the argument of Gagliardo \cite{Gagliardo} and Nirenberg \cite{Nirenberg}, that one can integrate in any direction and pair the gradient with an arbitrary covector (and we continue to restrict our consideration to the plane):   For any vector $v \in \Rset^2=(\Rset^2)^*$ and any vector $w \in \Rset^2$ one has
\begin{align*}
|\langle v, u(x_1,x_2)\rangle| \leq \int_{\Rset}  \left|  \langle v, Du(x+tw)w\rangle \right|\dif t.
\end{align*}
Here again the choices \(
w=v=(1,1)\) and \(w=v=(1,-1)\)
lead to the inequalities
\begin{align*}
|(u_1+u_2)(x_1,x_2)| &\leq \int_{\Rset}  \left| \left(\frac{\partial u_1}{\partial x_1} + \frac{\partial u_2}{\partial x_2} +\frac{\partial u_1}{\partial x_2}  + \frac{\partial u_2}{\partial x_1}\right)(x_1+t,x_2+t) \right|\dif t, \\
\intertext{and}
|(u_1-u_2)(x_1,x_2)| &\leq \int_{\Rset}  \left| \left(\frac{\partial u_1}{\partial x_1}  + \frac{\partial u_2}{\partial x_2} -\left(\frac{\partial u_1}{\partial x_2}  + \frac{\partial u_2}{\partial x_1}\right)  \right)(x_1+t,x_2-t) \right|\dif t.
\end{align*}
Making a translation in $t$, one observes that the integrals on the right-hand-side depend only on $x_1-x_2$ and $x_1+x_2$, respectively.  Letting $x_1-x_2=s \in \mathbb{R}$ in the former and $x_1+x_2=s \in \mathbb{R}$ in the latter, for each such 
$s$ we can take the supremum over all such pairs $(x_1,x_2)$ and then integrate in $s$ to obtain
\begin{align*}
\int_{\Rset} \sup_{x_1-x_2=s}|(u_1+u_2)(x_1,x_2)| \dif s &\leq \int_{\Rset^2}  \left| \left(\frac{\partial u_1}{\partial x_1} + \frac{\partial u_2}{\partial x_2} +\frac{\partial u_1}{\partial x_2}  + \frac{\partial u_2}{\partial x_1}\right)(t,t-s) \right|\dif t\dif s \\
\intertext{and}
\int_{\Rset} \sup_{x_1+x_2=s}|(u_1-u_2)(x_1,x_2)| \dif s &\leq \int_{\Rset^2}  \left| \left(\frac{\partial u_1}{\partial x_1}  + \frac{\partial u_2}{\partial x_2} -\left(\frac{\partial u_1}{\partial x_2}  + \frac{\partial u_2}{\partial x_1}\right)  \right)(t,s-t) \right|\dif t \dif s.
\end{align*}
It only remains to change variables to see that the claim has been demonstrated.
\end{proof}

\section{Gagliardo's lemma and Loomis--Whitney inequality}\label{gagliardo}

Our proof is based on a geometric inequality between the measure of a set and the measure of its projections on hyperplanes which goes back to Loomis and Whitney \cite{Loomis_Whitney_1949}.  To state the following generalization of their inequality, we require the notion of \((n -1)\)--dimensional Hausdorff measure of a set $A$, which we  denote by \(\mathcal{H}^{n - 1} (A)\), see e.g.  \cite{Evans_Gariepy_1992}*{\S 2.1, p.~60}.

\begin{lemma}%
[Loomis--Whitney inequality]
\label{lemma_Loomis_Whitney}
Let \(w_1, \dotsc, w_n\) be a basis of unit-length vectors of \(\Rset^n\) and let, for every \(j \in \{1, \dotsc, n\}\),
\(\Pi_j : \Rset^n \to \Rset^n\) denote the orthogonal projection of 
\(\Rset^n\) on \(w_j^\perp\). 
Then there exists a constant such that for every compact set \(K \subset \Rset^n\),
\[
    \mathcal{L}^n (K)^{n - 1}
  \le
    \frac{\mathcal{H}^{n - 1} \bigl(\Pi_1 (K)\bigr)
    \dotsm
    \mathcal{H}^{n - 1} \bigl(\Pi_n (K)\bigr)}
    {\abs{\det (w_1, \dotsc, w_n)}}.
\]
\end{lemma}

In the two-dimensional plane, the constant appearing in the inequality  corresponds geometrically to the absolute value of the sine of the angle between the vectors \(w_1\) and \(w_2\), while in any number of dimensions we have that equality is achieved in \cref{lemma_Loomis_Whitney} when \(K\) is a parallelepiped spanned by the vectors \(w_1, \dotsc, w_n\).
   
The original statement of Loomis and Whitney assumes that the vectors \(w_1, \dotsc, w_n\) are the canonical basis of \(\Rset^n\) \cref{lemma_Loomis_Whitney} and is proved by a combinatorial argument through an approximation by sets that are a finite collection of cubes.  Our approach shows how not only can one obtain the result of Loomis and Whitney as a direct consequence of the particular case of characteristic functions of a later lemma of Gagliardo \cite{Gagliardo}*{lemma 4.1}, in fact one easily obtains in a geometric fashion the preceding more general version of their result.

%Loomis and Whitney presented \cref{lemma_Loomis_Whitney} as a cheap version of the isoperimetric inequality: if the projections are bounded by the area of the boundary, this gives an isoperimetric inequality with a non-optimal constant. Our work shows on the contrary, that the application of \cref{lemma_Loomis_Whitney} to Sobolev estimates yields sharper estimates on the scale of Lorentz spaces than the classical deduction of the Sobolev embedding from the isoperimetric theorem \citelist{\cite{Mazya_1960}\cite{Federer_Fleming_1960}}.

\begin{lemma}
\label{lemma_Gagliardo_prime}
Let $n \geq 2$. Let \(P_i\) denote the canonical projection of \(\Rset^n\) on \(\Rset^{n - 1} \simeq \Rset^{i - 1} \times \{0\} \times \Rset^{n - i} \subset \Rset^n\).
For every choice of $f_i \in L^{n-1} (\Rset^{n-1})$, $i\in \{1,\ldots, n\}$, one has $\prod_{i=1}^{n} f_i \compose P_i \in L^1(\Rset^n)$ with the estimate
\[
  \biggabs{\int_{\Rset^n}
      \prod_{i = 1}^{n}
        f_i \compose P_i}
  \le  
  \prod_{i = 1}^{n }
 \biggl(
      \int_{\Rset^{n-1}}
        \abs{f_i}^{n - 1}
    \biggr)^\frac{1}{n - 1}.
\]
\end{lemma}

The proof is the classical proof of Gagliardo that we give here for the convenience of the reader.

\begin{proof}[Proof of \cref{lemma_Gagliardo_prime}]
We proceed by induction.  First let us treat the base case, $n=2$.  In this case, from an application of Fubini's theorem we find
\[
\begin{split}
    \biggabs{\int_{\Rset^2} 
      (f_1 \compose P_1)\,(f_2 \compose P_2)}
  &\le
 \int_{\Rset^2} \abs{f_1(z_2)} \, \abs{f_2(z_1)} \dif z_1 \dif z_2\\
     &=
 \biggl(\int_{\Rset} \abs{f_1} \biggr)
 \biggl(\int_{\Rset} \abs{f_2} \biggr).
 \end{split}
\]

Thus we proceed to the general case.  For $n \geq 3$, we assume the lemma has been proved for $n-1$ and will prove it for $n$.  By Fubini's theorem we have
\[
    \int_{\Rset^n} 
      \prod_{i = 1}^n
      f_i \compose P_i
  =
    \int_{\Rset^{n-1}}
        \biggl(\,\int_{\Rset}
              \prod_{i = 1}^{n-1} 
              f_i \bigl(P_i (z)\bigr) 
          \dif z_n
        \biggr)
        \,
       f_n (z')
      \dif z'.
\]
From two applications of H\"older's inequality successively on \(\Rset^{n - 1}\) and on \(\Rset\) we deduce 
\begin{equation}
\label{eq_fc24e048cf}
\begin{split}
    \biggl\lvert \int_{\Rset^n} &
      \prod_{i = 1}^n
      f_i \compose P_i \biggr\rvert\\
  &\le
  \left(\int_{\Rset^{n-1}}
        \biggabs{\int_{\Rset}
              \prod_{i = 1}^{n-1} 
              f_i \bigl(P_i (z', z_n)\bigr) \dif z_n\,
        }^\frac{n-1}{n-2} \dif z'\right)^\frac{n-2}{n-1}
        \left(\int_{\Rset^{n-1}} \abs{f_n}^{n - 1} \right)^\frac{1}{n-1}
      \\
    &\le
  \left(\int_{\Rset^{n-1}}
       \biggl(\,
       \prod_{i = 1}^{n-1}  
          \int_{\Rset}
              \bigabs{f_i \bigl(P_i (z', z_n)\bigr)}^{n-1}
          \dif z_n
        \biggr)^\frac{1}{n-2} \dif z'\right)^\frac{n-2}{n-1}
        \left(\int_{\Rset^{n-1}} \bigabs{f_n}^{n - 1} \right)^\frac{1}{n-1}.
\end{split}
\end{equation}
We now work to apply our induction assumption. 
For each \(i \in \{1, \dotsc, n - 1\}\), we define the function \(g_i : \Rset^{i - 1} \times \{0\} \times \Rset^{n - 1 - i} \times \{0\} \simeq \Rset^{n - 2} \to \Rset\) for each \(y \in\Rset^{i - 1} \times \{0\} \times \Rset^{n - 1 - i} \times \{0\}\)  by 
\begin{equation}
\label{eq_oadaXo5iew}
    g_i (y) 
  \defeq
    \biggl(
      \int_{\Rset} 
        \bigabs{
          f_i\left(y+ (0, \dotsc,0, z_n)\right)
        }^{n-1}
        \dif z_n
    \biggr)^\frac{1}{n-2},
\end{equation}
so that for every \(z' \in \Rset^{n -1}\simeq \Rset^{n - 1} \times\{0\} \subset \Rset^n\),
\[
  g_i \bigl(P_i (z')\bigr) 
  =
    \biggr(
      \int_{\Rset}
          \bigabs{f_i \bigl(P_i (z', z_n)\bigr)}^{n - 1}
        \dif z_n
    \biggr)^\frac{1}{n-2}.
\]
We observe that by Fubini's theorem
\begin{equation}
\label{eq_leeTieQu6H}
 \begin{split}
      \int_{\Rset^{n-2}} g_i{}^{n - 2} 
  &=  
    \int_{\Rset^{n-2}}  
      \int_{\Rset} 
        \bigabs{f_i (y+ (0, \dotsc, z_n))}^{n-1} \dif z_n \dif y\\
    &
    = \int_{\Rset^{n - 1}} \abs{f_i}^{n - 1} < + \infty,
 \end{split}
\end{equation}
so that $g_i \in L^{n-2}(\Rset^{n-2})$.  
Therefore we may apply our induction assumption to deduce
\begin{align}
\label{eq_21ac438f02}
  \int_{\Rset^{n-1}}
      \prod_{i = 1}^{n - 1}
        (g_i \compose P_i) 
 \le 
\prod_{i = 1}^{n - 1} 
    \biggl(
      \int_{\Rset^{n-2}}
        g_i{}^{n - 2}
    \biggr)^\frac{1}{n - 2}.
\end{align}
Putting these inequalities \eqref{eq_fc24e048cf}, \eqref{eq_oadaXo5iew} and \eqref{eq_21ac438f02} together we find
\begin{equation*}
\begin{split}
   \biggl\lvert \int_{\Rset^n} &
      \prod_{i = 1}^n
      f_i \compose P_i\biggr\rvert \\
  &\le
   \biggl(\int_{\Rset^{n-1}}
       \prod_{i = 1}^{n-1}  \biggl(\int_{\Rset}
              \bigabs{f_i \bigl(P_i (z', z_n)\bigr)}^{n-1}
          \dif z_n
        \biggr)^\frac{1}{n-2} \dif z'\biggr)^\frac{n-2}{n-1}
        \biggl(\int_{\Rset^{n-1}} \bigabs{f_n}^{n - 1} \biggr)^\frac{1}{n-1}\\
  &=  \biggl(\int_{\Rset^{n-1}}
       \prod_{i = 1}^{n-1}  g_i \compose P_i \biggr)^\frac{n-2}{n-1}
        \biggl(\int_{\Rset^{n-1}} \bigabs{f_n}^{n - 1} \biggr)^\frac{1}{n-1}\\
  &\le 
 \prod_{i = 1}^{n-1}
       \biggl(
      \int_{\Rset^{n-2}}
        g_i{}^{n - 2}
    \biggr)^\frac{1}{n - 1}\biggl(\int_{\Rset^{n-1}} \bigabs{f_n}^{n - 1} \biggr)^\frac{1}{n-1} \\
      &=
       \prod_{i = 1}^{n} \biggl(\int_{\Rset^{n-1}} \abs{f_i}^{n - 1}
      \biggr)^\frac{1}{n-1},
      \end{split}
      \end{equation*}
      in view of the identity \eqref{eq_leeTieQu6H}, 
      which is the thesis.
\end{proof}

\begin{proof}[Proof of \cref{lemma_Loomis_Whitney}]
Let $\{w_1,\ldots w_n\}$ be a basis of $\Rset^n$.
We define for every \(i \in \{1, \dotsc, n\}\) the function 
\(f_i : \Rset^{i-1} \times \{0\} \times \Rset^{n - i} \to \Rset\) for each \(z' \in \Rset^{i-1} \times \{0\} \times \Rset^{n - i}\) by
\[
    f_i (z') 
  \defeq
    \sup_{t \in \mathbb{R}} \chi_{K}\left(\textstyle \sum_{j\neq i} z_jw_j +tw_i\right),
\]
where \(\chi_K : \Rset^n \to \Rset\) is the characteristic function of the set \(K\).  Then we observe that for any $z \in \mathbb{R}^n$ we have
\[
\chi_K \Bigl({\textstyle \sum_{j = 1}^n z_j w_j}\Bigr) \leq  f_i (P_i(z)),
 \]
 and, as both sides assume only the values \(0\) and \(1\), we have
 \[
\chi_K \Bigl({\textstyle \sum_{j = 1}^n z_j w_j}\Bigr) \leq  \prod_{i=1}^n f_i (P_i(z)).
 \]
It follows thus that 
\[
\begin{split}
  \mathcal{L}^n (K)
  &=
  \abs{\det (w_1, \dotsc, w_n)}
   \int_{\Rset^n}
    \chi_K \Bigl({\textstyle \sum_{j = 1}^n z_j w_j}\Bigr)
    \dif z\\
 &\le
  \abs{\det (w_1, \dotsc, w_n)} 
 \int_{\Rset^n}
  \prod_{i = 1}^n f_i (P_i(z)) 
  \dif z.
\end{split}
\]
We observe now that by Gagliardo's inequality (\cref{lemma_Gagliardo_prime}) we have 
\[
\int_{\Rset^n}
      \prod_{i = 1}^{n}
        f_i \compose P_i 
  \le 
  \prod_{i = 1}^{n }
    \bigg(
     \int_{\Rset^{i-1} \times \{0\} \times \Rset^{n - i}}
       \hspace{-1em}
        \abs{f_i}^{n-1}
    \biggr)^\frac{1}{n - 1}.
\]
But now for \(z' \in \Rset^{i-1} \times \{0\} \times \Rset^{n - i}\) we have
\[
f_i(z') = \sup_{t \in \mathbb{R}} \chi_{K}\biggl(\sum_{j\neq i} z_jw_j +tw_i\biggr)\\
= (\chi_{\Pi_i(K)} \compose \Pi_i)\biggl(\sum_{j \neq i} z_jw_j \biggr),
\]
while
\[
 \mathcal{H}^{n - 1} \bigl(\Pi_i (K)\bigr)
 = J_i
 \int_{\Rset^i \times \{0\} \times \Rset^{n - i}}
        (\chi_{\Pi_i(K)} \compose \Pi_i) \Bigl(\textstyle\sum_{j \neq i} z_jw_j \Bigr)
        \dif z',
\]
where $J_i$ is the (constant) Jacobian of the linear map $z' \in \Rset^i \times \{0\} \times \Rset^{n - i} \mapsto \Pi_i \bigl(\sum_{j\neq i} z_jw_j\bigr)$ (see \cite{Evans_Gariepy_1992}*{\S 3.2}).
We now compute this Jacobian: if \(i = 1\), we have since \(\abs{w_1} = 1\), by elementary manipulations of lines and columns of determinants 
\[
\begin{split}
  J_1{}^2
  &=
  \det
    \begin{pmatrix}
     (w_2 \cdot w_2) - (w_1 \cdot w_2)(w_1 \cdot w_2) & \hdots &
      (w_2 \cdot w_n) - (w_1 \cdot w_2)(w_1 \cdot w_n)\\
      \vdots & \ddots & \vdots\\
     (w_n \cdot w_2) - (w_1 \cdot w_n)(w_1 \cdot w_2) & \hdots &
     (w_n \cdot w_n) - (w_1 \cdot w_n)(w_1 \cdot w_n)\\
          \end{pmatrix}\\
  &=
  \det
    \begin{pmatrix}
     1 & 0 & \hdots & 0\\
     0 &(w_2 \cdot w_2) - (w_1 \cdot w_2)(w_1 \cdot w_2) & \hdots &
      (w_2 \cdot w_n) - (w_1 \cdot w_2)(w_1 \cdot w_n)\\
     \vdots& \vdots & \ddots & \vdots\\
     0 & (w_n \cdot w_2) - (w_1 \cdot w_n)(w_1 \cdot w_2) & \hdots &
     (w_n \cdot w_n) - (w_1 \cdot w_n)(w_1 \cdot w_n)\\
          \end{pmatrix}\\          
  &=
  \det
    \begin{pmatrix}
     1 & (w_1 \cdot w_2) & \hdots & (w_1 \cdot w_n)\\
     0 &(w_2 \cdot w_2) - (w_1 \cdot w_2)(w_1 \cdot w_2) & \hdots &
      (w_2 \cdot w_n) - (w_1 \cdot w_2)(w_1 \cdot w_n)\\
     \vdots& \vdots & \ddots & \vdots\\
     0 & (w_n \cdot w_2) - (w_1 \cdot w_n)(w_1 \cdot w_2) & \hdots &
     (w_n \cdot w_n) - (w_1 \cdot w_n)(w_1 \cdot w_n)\\
          \end{pmatrix}\\
  &=
  \det
    \begin{pmatrix}
     1 & (w_1 \cdot w_2) & \hdots & (w_1 \cdot w_n)\\
     (w_2 \cdot w_1) &(w_2 \cdot w_2)  & \hdots &
      (w_2 \cdot w_n) \\
     \vdots& \vdots & \ddots & \vdots\\
     (w_n \cdot w_1) & (w_n \cdot w_2)  & \hdots &
     (w_n \cdot w_n) \\
    \end{pmatrix}\\
      &=
  \det
    \begin{pmatrix}
     (w_1 \cdot w_1) &  \hdots & (w_1 \cdot w_n)\\
     \vdots& \ddots & \vdots\\
     (w_n \cdot w_1) &  \hdots &
     (w_n \cdot w_n)\\
    \end{pmatrix} = 
    \abs{\det (w_1, \dotsc, w_n)}^2
\end{split}
\]
(this computation is in fact a case of computation of determinant through the Schur complement \cite{Horn_Johnson_2013}*{(0.8.5.1)}); the case \(i \in \{2, \dotsc, n\}\) is similar and therefore
\[
\begin{split}
  \mathcal{L}^n (K)
  &\le  \abs{\det (w_1, \dotsc, w_n)}
  \left(
  \,
 \prod_{i = 1}^{n }
     \int_{\Rset^{i-1} \times \{0\} \times \Rset^{n - i}}
       \hspace{-1em}
         \chi_{\Pi_i(K)}\left(\textstyle \sum_{j \neq i} z_jw_j \right) \dif z'\right)^\frac{1}{n - 1}\\
     &= 
       \abs{\det (w_1, \dotsc, w_n)}
       \left( 
       \,
       \prod_{i = 1}^{n } 
         \frac
           {\mathcal{H}^{n - 1} \bigl(\Pi_i (K)\bigr)}
           {\abs{\det (w_1, \dotsc, w_n)}}\right)^\frac{1}{n - 1} \\
         &= \left(\frac{\prod_{i = 1}^{n } \mathcal{H}^{n - 1} \bigl(\Pi_i (K)\bigr)}
         {\abs{\det (w_1, \dotsc, w_n)}}\right)^\frac{1}{n - 1},
\end{split}
\]
thus concluding the demonstration of the claim.
\end{proof}

\section{Proofs of the Main Results}

\subsection{Estimates by directional derivatives of components}

\label{proofs}

The last tool that we will need in the proofs is an estimate on the norm by sets of projections.

\begin{lemma}
\label{lemma_Estimate_MLI}
Under the assumptions of \cref{proposition_directional_derivatives},
there exists a constant \(C \in \Rset\) 
such that for every \(v \in V\),
one has
\[
    \abs{v}
  \le
    C
    \max
      \,
      \Bigl\{ 
          \min\, \bigl\{ 
            \abs{\dualprod{v_1^j}{v}}, 
            \dotsc, 
            \abs{\dualprod{v_n^j}{v}} \bigr\}
        \st
          j \in \{1, \dotsc, \ell\}
       \Bigr\}
      .
\]
\end{lemma}
\begin{proof}
Let \(\gamma : V \to \Rset\) denote the function defined so that for every \(v \in V\),
the value \(\gamma (v) \in \Rset\) is the right-hand side of the conclusion.
The function \(\gamma\) is nonnegative, continuous and positively homogeneous of degree \(1\).
We will reach the conclusion by proving that the function 
\(\gamma\) only vanishes at  the point \(0\): 
indeed, if \(v \in V\) and \(\gamma (v) = 0\),
then for every \(j \in \{1, \dotsc, \ell\}\), 
we have
\(v \in \bigcup_{i = 1}^n (v_i^j)^\perp\), and thus by assumption \(v = 0\).
\end{proof}

\begin{proof}%
[Proof of \cref{proposition_directional_derivatives}]
\resetconstant
For every \(t > 0\), we have by \cref{lemma_Estimate_MLI},
\[
    \bigl\{ x \in \Rset^n \st \abs{u (x)} \ge \Cl{MLI} t\,\bigr\}
  \subseteq
    \bigcup_{j = 1}^\ell
      \bigcap_{i = 1}^n 
        \bigl\{ x \in \Rset^n \st \abs{\dualprod{v_i^j}{u (x)}} \ge t\bigr\},
\]
for some constant \(\Cr{MLI} > 0\).
We deduce then by subadditivity of the measure and of the map \(\mu \in (0, + \infty) \mapsto \mu^{1 - \frac{1}{n}}\) that 
\begin{equation}
\label{eq_ZTQ3NWZhMD}
\begin{split}
    \mathcal{L}^n \bigl(\{ x \in \Rset^n \st \abs{u (x)} \ge \Cr{MLI} t\}\bigr)^{1 - \frac{1}{n}}
  &
  \le
    \Biggl(
    \sum_{j = 1}^\ell
      \mathcal{L}^n
      \biggl(
      \bigcap_{i = 1}^n 
        \{ 
          x \in \Rset^n 
        \st 
          \abs{\dualprod{v_i^j}{u (x)}} \ge t
        \}
       \biggr)
     \Biggr)^{1 - \frac{1}{n}}
  \\
  &
  \le
    \sum_{j = 1}^\ell
      \mathcal{L}^n
      \biggl(
      \bigcap_{i = 1}^n 
        \{ 
          x \in \Rset^n 
        \st 
          \abs{\dualprod{v_i^j}{u (x)}} \ge t
        \}
      \biggr)^{1 - \frac{1}{n}}.
\end{split}
\end{equation}
If \(j \in \{1, \dotsc, \ell\}\), 
then by assumption, the vectors \(w_1^j, \dotsc, w_n^j\) are linearly independent in \(\Rset^n\) 
and thus by
the Loomis--Whitney inequality (\cref{lemma_Loomis_Whitney}) and by monotonicity of the measure, we have 
\begin{equation}
\label{eq_gICAgLS4K}
\begin{split}
 \mathcal{L}^n
      \biggl(
      \bigcap_{i = 1}^n 
        &\bigl\{ 
          x \in \Rset^n 
        \st 
          \abs{\dualprod{v_i^j}{u (x)}} \ge t
        \bigr\}
      \biggr)^{1 - \frac{1}{n}}\\[-.8em]
  &\le  
    \Cl{c_m_LW}
    \prod_{k = 1}^n 
      \mathcal{H}^{n - 1} 
        \Biggl(\Pi_k^j 
          \biggl(
          \,
            \bigcap_{i = 1}^n 
              \bigl\{ x \in \Rset^n \st \abs{\dualprod{v_i^j}{u (x)}} \ge t\bigr\}
            \biggr)
        \Biggr)^\frac{1}{n}\\
  &\le  
    \Cr{c_m_LW}
    \prod_{i = 1}^n 
      \mathcal{H}^{n - 1} 
        \Bigl(\Pi_i^j 
          \bigl(\bigl\{ x \in \Rset^n \st \abs{\dualprod{v_i^j}{u (x)}} \ge t\bigr\}\bigr)
        \Bigr)^\frac{1}{n},
\end{split}
\end{equation}
where \(\Pi_i^j : \Rset^n \to \Rset^n\) is the orthogonal projection on the hyperplane \(W_i^j \defeq w_i^j{}^\perp = \Pi_i^j (\Rset^n)\).
Finally, we observe that if 
\[
    y 
  \in 
    \Pi_i^j 
      \bigl(\{ x \in \Rset^n \st \abs{\dualprod{v_i^j}{u (x)}} \ge t\}\bigr),
\]
then there exists a real number \(h \in \Rset\) such that \(\abs{\dualprod{v_i^j}{u (y + h w_{i}^j)}} \ge t\) and thus 
\[
 \int_{\Rset} \abs{\dualprod{v_i^j}{D u (y + s w_{i}^j)[w_{i}^j]}} 
 \dif s \ge
 2 \,\abs{\dualprod{v_i^j}{u (y + h w_{i}^j)}}
 \ge 2 t.
\]
If we define the function \(F_i^j: W_i^j \to \Rset\) by  setting for each \(y \in W_i^j\)
\[
    F_i^j (y) 
  \defeq 
    \frac{1}{2} 
    \int_{\Rset} 
      \abs{\dualprod{v_i^j}{D u (y + s w_{i}^j)[w_{i}^j]}}\dif s,
\]
we have for each \(t > 0\)
\[
 \Pi_i^j 
          \bigl(\bigl\{ x \in \Rset^n \st \abs{\dualprod{v_i^j}{u (x)}} \ge t\bigr\}\bigr)
          \subseteq
       \bigl\{ 
          y \in W_i^j
        \st 
          F_i^j (y) \ge t
        \bigr\}
\]
and thus by \eqref{eq_gICAgLS4K}
\[
  \mathcal{L}^n
      \biggl(
      \,
      \bigcap_{i = 1}^n 
        \{ 
          x \in \Rset^n 
        \st 
          \abs{\dualprod{v_i^j}{u (x)}} \ge t
        \}
      \biggr)^{1 - \frac{1}{n}}
  \le 
    \Cr{c_m_LW}
    \prod_{i = 1}^n 
      \Bigl(
        \mathcal{H}^{n-1} 
        (\bigl\{ 
          y \in W_i^j
        \st 
          F_i^j (y) \ge t
        \bigr\}
      \Bigr)^\frac{1}{n}.
\]
In view of \eqref{eq_ZTQ3NWZhMD} and of the H\"older inequality,  we obtain
\begin{equation*}
\begin{split}
    \int_0^\infty 
      \mathcal{L}^n 
        \bigl(
          \bigl\{ 
            x \in \Rset^n 
          &\st 
            \abs{u (x)} \ge \Cr{MLI} t
          \bigr\}
        \bigr)^{1 - \frac{1}{n}} 
      \dif t
      \\[-.8em]
  &\le 
    \Cl{c_m_final}
    \sum_{j = 1}^\ell
    \prod_{i = 1}^n 
      \Biggl(
        \int_0^\infty 
          \mathcal{H}^{n-1} 
            \bigl(
              \bigl\{ 
                y \in W_i^ j
              \st 
                F_i^j (y) \ge t
              \bigr\} 
          \dif t 
      \Biggr)^\frac{1}{n}
      \\
  &=
    \Cr{c_m_final}
        \sum_{j = 1}^\ell
    \prod_{i = 1}^n 
      \Biggl(
        \int_{W_i^j}
          F_i^j
          \dif \mathcal{H}^{n - 1}
      \Biggr)^\frac{1}{n}\\
  &= \frac{\Cr{c_m_final}}{2}
    \sum_{j = 1}^\ell
      \prod_{i = 1}^n 
        \Biggl(
          \int_{\Rset^n} 
            \abs{\dualprod{v_i^j}{D u[w_i^j]}}
        \Biggr)^\frac{1}{n}. \qedhere
\end{split}
\end{equation*}
\end{proof}

\subsection{Direct consequences}

\label{section_sirect_consequences}

We firstly show of \eqref{Alvino} can be deduced from \cref{proposition_directional_derivatives}.

\begin{proof}[Proof of  \eqref{Alvino} by \cref{proposition_directional_derivatives}]
\resetconstant
Let \(v_1, \dotsc, v_m\) be a basis of \(V^*\) and \(w_1, \dotsc, w_n\) be a basis of \(\Rset^n\). We set \(\ell = m\), and for every \(i \in \{1, \dotsc, n\}\) and \(j \in \{1, \dotsc, m\}\), \(v_i^j = v_j\) and \(w_i^j = w_i\). 
We have 
\[
    \bigcap_{j = 1}^m 
      \bigcup_{i = 1}^n 
        (v_i^j)^\perp  
  = 
    \bigcap_{j = 1}^m 
      v_j^\perp
  =   
    \{0\},
\]
so that \cref{proposition_directional_derivatives} applies and we conclude by Young's inequality and by norm equivalence that 
\[
\begin{split}
    \norm[L^{\frac{n}{n - 1}, 1}(\Rset^n,V)]{u}
  &\le
    \Cl{cst_diso}
    \sum_{j = 1}^m
      \prod_{i = 1}^n
        \norm[L^1(\Rset^n)]{\dualprod{v_j}{D u[w_i]}}^\frac{1}{n}
        \\
  &\le
    \frac{\Cr{cst_diso}}{n}
    \sum_{j = 1}^m
      \sum_{i = 1}^n
        \norm[L^1(\Rset^n)]{\dualprod{v_j}{D u[w_i]}}
  \le
    \C
    \norm[L^1(\Rset^n,\Rset^n)]{D u}.\qedhere
\end{split}
\]
\end{proof}

We next prove an analogue of De Figueiredo's $L^2$ inequality, from which we can deduce the improvement to Strauss' Korn-Sobolev inequality.  To this end it will be useful to introduce the following definition.

\begin{definition}
If \(V\) is a vector space, then a finite set of vectors \(F \subset V\) is \emph{maximally linearly independent}, 
whenever for every subset \(A \subset F\), either \(A\) generates \(V\) as a linear space or \(A\) is linearly independent.
\end{definition}

\begin{lemma}
\label{lemma_de_Figueiredo}
If \(n \ge 1\), \(\dim V = m\) and if \(v_1, \dotsc, v_{n + m - 1}\) are maximally linearly independent in \(V^*\), then 
\[
    \bigcap_{1 \le i_1 < \dotsb < i_n \le n + m - 1}
      \Biggl(
        \bigcup_{j = 1}^n v_{i_j}{}^\perp
      \Biggr)
  =
    \{0\}
  .
\]
\end{lemma}

\begin{proof}
Assume that \(v\) belongs to the left hand side and let \(I = \{i \in \{1, \dotsc, n + m - 1\} \st \dualprod{v_i}{v} = 0\}\). We have then \(\# I \ge m\). Indeed, otherwise there would exist \(i_1, \dotsc, i_n \in \{1, \dotsc, n - m - 1\}\)
such that \(1 \le i_1 < \dotsb < i_{n} \le n + m - 1 \) and for every \(j \in \{1, \dots, n\}\), 
\(\dualprod{v_{i_j}}{v} \ne 0\) and so \(v \not \in \bigcup_{j = 1}^n v_{i_j}{}^\perp\), in contradiction with our assumption. Since the family \(v_1, \dotsc, v_{n + m - 1}\) is maximally linearly independent, the set \(\{v_i \st i \in I\}\) generates the \(m\)--dimensional linear space \(V\) and thus \(v = 0\).
\end{proof}

\begin{theorem}
\label{proposition_de_Figueiredo}
Assume that \(\dim V = m\) and that the vectors \(w_1, \dotsc, w_{n + m - 1} \in \Rset^n\)
and \(v_1, \dotsc, v_{n + m - 1} \in V\) are maximally linearly independent, 
then for every \(u \in C^\infty_c (\Rset^n, V)\),
\[
    \norm[L^{\frac{n}{n - 1}, 1}(\Rset^n,V)]{u}
  \le
    \hspace{-1em}
    \sum_{1 \le i_1 < \dotsb < i_n \le n + m - 1} 
    \hspace{-1em}
      \norm[L^1(\Rset^n)]{\dualprod{v_{i_1}}{  D u[w_{i_1}]}}^\frac{1}{n}
%        \norm{\dualprod{v_{i_2}}{  D u[w_{i_2}]}}_{L^1(\Rset^n)}^\frac{1}{n}
      \dotsm
       \norm[L^1(\Rset^n)]{\dualprod{v_{i_n}}{  D u[w_{i_n}]}}^\frac{1}{n}
  .    
\]
\end{theorem}

These sparse directional Sobolev estimates into Lorentz space are analogous to \(L^2\) estimates of de Figueiredo \cite{DeFigueiredo}
and strengthen known results for Sobolev estimates into \(L^\frac{n}{n - 1}\) \citelist{\cite{Strauss}\cite{Bourgain_Brezis_2007}*{Remark 16}\cite{VanSchaftingen_2013}*{Proposition 6.8}}.

\begin{proof}%
[Proof of \cref{proposition_de_Figueiredo}]
This follows from \cref{proposition_directional_derivatives} and \cref{lemma_de_Figueiredo}.
\end{proof}

Finally, we can utilize the preceding inequality to deduce the Korn--Sobolev inequality.

\begin{proof}%
[Proof of \cref{proposition_Korn_Sobolev} by \cref{proposition_de_Figueiredo}]
\resetconstant
We consider \(w_1, \dotsc, w_{2n - 1}\) to be a maximally independent family of vectors of \(\Rset^n\).
We observe now that, since \(Eu\) is the symmetric part of \(Du\),
\[
 \abs{\dualprod{w_j}{D u[w_j]}}
 = \abs{w_j \cdot E u [w_j]} \le \abs{w_j}^2 \abs{E u}
\]
and the conclusion then follows from \cref{proposition_de_Figueiredo}.
\end{proof}

\subsection{Estimates for \((n-1)\)--canceling operators}
\label{canceling}

In order to set an algebraic condition on differential operators, we introduce a new scale of conditions on differential operators that covers the definition of canceling operators \cite{VanSchaftingen_2013}*{Definition 1.2}.

\begin{definition}
Let \(\ell \in \{0, \dotsc, n\}\).
A homogeneous differential operator with constant coefficients \(A (D)\) is \emph{ \(\ell\)--canceling} whenever 
\[
    \bigcap_{\substack{W \subseteq \Rset^n\\ \dim W = \ell}}
      \linspan 
        \,
        \bigl\{ 
          A (\xi)[v] 
        \st 
          \xi \in W 
          \text{ and } 
          v \in V
        \bigr\}
  =
    \{0\}
  .
\]
\end{definition}

An operator is \(1\)--canceling if and only if it is canceling in the sense of \cite{VanSchaftingen_2013}*{Definition 1.2}.
Any operator \(A (D)\) is \(0\)--canceling; an operator \(A (D)\) is \(n\)--canceling if and only if \(A (D) = 0\). 

Optimal estimates into Lorentz spaces hold under the \((n-1)\)--canceling condition:

\begin{theorem}
\label{theorem_n_1_canceling}
Let $V$ and $E$ be finite-dimensional spaces and suppose that the homogeneous linear differential operator with constant coefficients \(A (D):C^\infty_c (\Rset^n, V) \to C^\infty_c (\Rset^n, E)\) is elliptic and \((n - 1)\)--canceling.  Then there exists a constant $C>0$ such that 
\[
    \norm[L^{\frac{n}{n - 1},1}(\Rset^n,V)]{u} 
  \le 
    C
    \norm[L^1(\Rset^n,E)]{A (D) u}
\]
for every \(u \in C^\infty_c (\Rset^n, V)\).
\end{theorem}

We will deduce \cref{theorem_n_1_canceling} from \cref{proposition_directional_derivatives} with the help of the next algebraic lemma.

\begin{lemma}
\label{lemma_structure_n_1_canceling}
If the first-order homogeneous differential operator  with constant coefficients \(A (D)\) is elliptic and \((n-1)\)--canceling, then there exists \(m \in \Nset\) and, for each \(j \in \{1, \dotsc, m\}\) and \(i \in \{1, \dotsc, n\}\), vectors \(w_i^j \in \Rset^n\), \(v_i^j \in V^*\) and \(e_i^j \in E^*\) such that 
\begin{enumerate}%
[(i)]
  \item for every \(j \in \{1, \dotsc, m\}\), the vectors \(w_1^j, \dotsc, w_n^j\) are linearly independent in \(\Rset^n\),
  \item 
    \(\displaystyle 
      \bigcap_{j = 1}^m
        \bigcup_{i = 1}^n (v_i^j)^\perp = \{0\}\),
  \item 
    for every \(j \in \{1, \dotsc, m\}\),
    \(i \in \{1, \dotsc, n\}\),
    \(\xi \in \Rset^n\)
    and 
    \(v \in V\),
    \[
         \dualprod{w_i^j}{\xi}  \dualprod{v_i^j}{v}
      =
        \dualprod{e_i^j}{A (\xi)[v]}.
    \]
\end{enumerate}
\end{lemma}

The proof of \cref{lemma_structure_n_1_canceling} will proceed by induction, the next lemma is the key step in the iteration.

\begin{lemma}
\label{lemma_vectors_1dir}
If the first-order homogeneous differential operator  with constant coefficients \(A (D)\) is elliptic and canceling, 
then for every \(v_* \in V \setminus \{0\}\), there exists vectors \(\xi_1, \dotsc, \xi_n \in \Rset^n\), vectors \(w_1, \dotsc, w_n \in \Rset^n\) and vectors \(e_1, \dotsc, e_n \in E^*\) such that 
\begin{enumerate}%
[(i)]
  \item 
    the vectors \(w_1, \dotsc, w_w\)
    are linearly independent in \(\Rset^n\),
  \item  
    for every \(i \in \{1, \dotsc, n\}\), 
    \(\dualprod{e_i}{A (\xi_i)[v_*]} = 1\),
  \item  
    for every \(i \in \{1, \dotsc, n\}\), 
    \(\dualprod{\xi_i}{w_i} = 1\),
  \item 
    for every \(i \in \{1, \dotsc, n\}\),
    \(\xi \in \Rset^n\)
    and \(v \in V\),
    \( 
        \dualprod{e_i}{A (\xi)[v]} 
      = 
        \dualprod{\xi}{w_i}
        \dualprod{A(\xi_i)^*[e_i]}{v} 
    \)%
    .
\end{enumerate}
\end{lemma}

\begin{proof}
We proceed by induction, that is, we are proving that for every \(\ell \in \{0, 1, \dotsc, n\}\),  there exists covectors \(\xi_1, \dotsc, \xi_\ell \in \Rset^n\), vectors \(w_1, \dotsc, w_\ell \in \Rset^n\) and vectors \(e_1, \dotsc, e_\ell \in E^*\) such that 
\begin{enumerate}%
[(a)]
  \item 
  \label{it_n1c_ind_linind}
    the vectors \(w_1, \dotsc, w_\ell \in \Rset^n\)
    are linearly independent in \(\Rset^n\),
  \item 
  \label{it_n1c_ind_nontriv}
    for every \(i \in \{1, \dotsc, \ell\}\), 
    \(\dualprod{e_i}{A (\xi_i)[v_*]} = 1\),
  \item  
  \label{it_n1c_ind_biorthog}
    for every \(i \in \{1, \dotsc, \ell\}\), 
    \(\dualprod{\xi_i}{w_i} = 1\),
  \item 
  \label{it_n1c_ind_repres}   
    for every \(i \in \{1, \dotsc, \ell\}\),
    \(\xi \in \Rset^n\)
    and \(v \in V\),
    \( 
        \dualprod{e_i}{A (\xi)[v]} 
      = 
         \dualprod{\xi}{w_i}
         \dualprod{A(\xi_i)^*[e_i]}{v}
    \)%
    .
\end{enumerate}

For \(\ell = 0\), the assertion holds vacuously. 
Assuming now that the assumption holds for some \(\ell \in \{0, 1, \dotsc, n - 1\}\), we will prove the assertion for \(\ell + 1\).

Since \(\ell \le n - 1\), there exists a covector \(\xi_{\ell + 1} \in \Rset^n \setminus \{0\}\) such that for every \(i \in \{1, \dotsc, \ell\}\), \(\dualprod{\xi_{\ell + 1}}{w_i}=0\). 
(In particular, if \(\ell = 0\), we just take any \(\xi_1 \in \Rset^n \setminus \{0\}\).)
Since the operator \(A (D)\) is elliptic, the linear operator \(A (\xi_{\ell + 1}) : V \to E\) is injective, and thus, since \(v_* \ne 0\), we have \(A (\xi_{\ell + 1})[v_*] \ne 0\). 
Since the operator \(A (D)\) is also \((n - 1)\)--canceling, there exists an \((n - 1)\)--dimensional linear subspace \(W_{\ell + 1} \subset \Rset^n\) such that 
\begin{equation}
\label{eq_luchieB0ie}
    A (\xi_{\ell + 1})[v_*] 
  \not \in 
    \linspan 
    \,
      \bigl\{ 
        A (\xi)[v] 
      \st 
          \xi \in W_{\ell + 1} 
        \text{ and } 
          v \in V
      \bigr\}.
\end{equation}
We define now \(w_{\ell + 1} \in \Rset^n\) to be a vector such that \(\dualprod{\xi_{\ell + 1}}{w_{\ell + 1}} = 1\) and for every \(\xi \in W_{\ell + 1}\), one has \(\dualprod{\xi}{w_{\ell + 1}} = 0\). 
In particular, this implies \eqref{it_n1c_ind_biorthog}. 
Since the vectors \(w_1, \dotsc, w_{\ell}\) are linearly independent and since by construction we have for every \(i \in \{1, \dotsc, \ell\}\), \(\dualprod{\xi_{\ell + 1}}{w_i} = 0\), the vectors \(w_1, \dotsc, w_\ell, w_{\ell + 1}\) are linearly independent in \(\Rset^n\) and thus \eqref{it_n1c_ind_linind} holds.
By \eqref{eq_luchieB0ie} there exists a covector \(e_{\ell + 1} \in E^*\) such that \(\dualprod{e_{\ell + 1}}{A (\xi_{\ell + 1})[v_*])} = 1\) and for every \(\xi \in W_{\ell + 1}\) and \(v \in V\), one has \(\dualprod{e_{\ell + 1}}{A (\xi)[v])} = 0\). In particular, \eqref{it_n1c_ind_nontriv} holds.
Moreover, we have \(A (\xi)^*[e_{\ell + 1}] = 0\) when \(\xi \in W_{\ell + 1}\). 
Therefore, since \(\dim W_{\ell + 1} = n - 1\), by the classical representation theorem of linear mappings, we deduce that for every \(\xi \in \Rset^n\),
\[
 A(\xi)^*[e_{\ell + 1}] = \dualprod{\xi}{w_{\ell + 1}}\, A (\xi_{\ell + 1})^*[e_{\ell + 1}],  
\]
which implies assertion \eqref{it_n1c_ind_repres}.
\end{proof}

A set \(X \subseteq V\) is a \emph{linear subspace arrangement} whenever \(X\) is a finite union of linear subspaces of \(V\).

\begin{lemma}
\label{lemma_descending}
Let \(V\) be a finite-dimensional space.
Assume that for each \(\ell \in \Nset\), \(X_\ell\) is a linear subspace arrangement of \(V\) and that \(X_\ell \supseteq X_{\ell + 1}\). Then there exists \(\ell_0 \in \Nset\) such that for every \(\ell \ge \ell_0\),
\(
    X_\ell 
  =
    X_{\ell_0}
\)%
.
\end{lemma}

\Cref{lemma_descending} can be proved by observing that for every \(\ell \in \Nset\), the set \(X_\ell\) is Zariski-closed so that the sequence \((X_\ell)_{\ell \in \Nset}\) is a descending chain of Zariski-closed sets for which the conclusion follows (see for example \cite{Hartshorne_1977}*{Example 1.4.7}). We give a direct proof for the convenience of the reader.

\begin{proof}
[Proof of \cref{lemma_descending}]
For every \(j\in \{1, \dotsc, \dim V\}\) we consider the set 
\[
    X_\ell^j
  =
    \bigcup 
      \,
      \bigl\{
        Y \subseteq V
      \st
        Y \text{ is a linear subspace, }
        \dim Y \ge j 
        \text{ and }
        Y \subseteq X_\ell
      \bigr\}.
\]
We have \(X_\ell^j \supseteq X_{\ell + 1}^j\) and \(X_\ell^j\) is a finite union of linear subspaces of dimension at least \(j\).

We prove now by downward induction, that for every \(j \in \{1, \dotsc, \dim V\}\), there exists \(\ell_j\) such that for every \(\ell \ge \ell_j\), we have
\(X_\ell^j = X_{\ell_j}^j\).
For \(j = \dim V\), either for every \(\ell \in \Nset\), \(X_\ell = V\) and then \(\ell_j = 0\),
or there exists \(\ell_j \in \Nset\) such that \(X_{\ell} = \{0\}\).

We assume now that the assertion is proved for some \(j \in \{2, \dotsc, \dim V\}\). 
We observe that for every \(\ell \ge \ell_j\), the components of \(X_\ell^{j - 1}\) of dimension at least \(j\) remain the same and the \((j - 1)\)--dimensional components of \(X_\ell^{j - 1}\) form a subset of those of \(X_{\ell_j}^{j - 1}\). 
We have thus a nonincreasing sequence of finite subsets; there exists thus \(\ell_{j - 1} \ge \ell_j\) such that for every \(\ell \ge \ell_{j - 1}\), \(X_\ell^{j - 1} = X_{\ell_{j - 1}}^{j -1}\).
\end{proof}

We are now in position to prove \cref{lemma_vectors_1dir}.

\begin{proof}%
[Proof of \cref{lemma_vectors_1dir}]
We are going to construct a the family of vectors iteratively over \(\ell\).
At each step, we assume that we have vectors \(w_i^j \in \Rset^n\), \(v_i^j \in V^*\) and \(e_i^j \in E^*\) for \(j \in \{1, \dotsc, \ell\}\) and \(i \in \{1, \dotsc, n\}\), such that 
\begin{enumerate}%
[(a)]
  \item for every \(j \in \{1, \dotsc, \ell\}\), the vectors \(w_1^j, \dotsc, w_n^j\) are linearly independent in \(\Rset^n\),
  \item 
    for every \(j \in \{1, \dotsc, \ell\}\),
    \(i \in \{1, \dotsc, n\}\),
    \(\xi \in \Rset^n\)
    and 
    \(v \in V\),
    \[
     \dualprod{w_i^j}{\xi} 
        \dualprod{v_i^j}{v}
      =
        \dualprod{e_i^j}{A (\xi)[v]}.
    \]
\end{enumerate}
This is trivially satisfied when \(\ell = 0\).

Assume thus that we have such families of vectors for some \(\ell\).
If 
\[
    X_\ell
  \defeq
    \bigcap_{j = 1}^\ell
      \bigcup_{i = 1}^n (v_i^j)^\perp = \{0\}, 
\]
then the proposition is proved with \(m = \ell\).
Otherwise, we take \(v_*^{\ell + 1} \in X_\ell \setminus \{0\}\), and we obtain by \cref{lemma_vectors_1dir}
vectors \(\xi_1^{\ell + 1},  \dotsc, \xi_n^{\ell + 1} \in \Rset^n\), \(w_1^{\ell + 1}, \dotsc, w_n^{\ell + 1} \in \Rset^n\) and \(e_1^{\ell + 1}, \dotsc, e_n^{\ell + 1} \in E^*\). We set \(v_i^{\ell + 1} = A (\xi_i^{\ell + 1})^*[e_i^{\ell + 1}]\), and we observe that the family satisfies the same condition and moreover, since for every \(i \in \{1, \dotsc, n\}\),  
\(
 \dualprod{v_i^{\ell + 1}}{v_*^{\ell + 1}} = 
 \dualprod{e_i^{\ell+ 1}}{A (\xi_i^{\ell + 1})[v_*^{\ell + 1}]} 
      =  1
\), we have \(v_*^{\ell + 1} \not \in X_{\ell + 1}\) 
and thus \(X_{\ell + 1} \subsetneq X_{\ell}\).

We conclude by observing that the procedure must finish after a finite number of steps in view of \cref{lemma_descending}.
\end{proof}

\begin{proof}%
[Proof of \cref{theorem_n_1_canceling}]
Let \(m \in \Nset\) and,  for each \(j \in \{1, \dotsc, m\}\) and \(i \in \{1, \dotsc, n\}\), the vectors \(w_i^j \in \Rset^n\), \(v_i^j \in V^*\) and \(e_i^j \in E^*\) be given for \(A (D)\) by \cref{lemma_structure_n_1_canceling}.
In view of \cref{proposition_directional_derivatives}, we have
\[
    \norm[L^{\frac{n}{n - 1}, 1}(\mathbb{R}^n,V)]{u}
  \le
    C
    \sum_{j=1}^m 
      \norm[L^1(\mathbb{R}^n)]{\dualprod{v_1^j}{D u[w_1^j]}}^\frac{1}{n}
%       \norm{\dualprod{v_2^j}{D u[w_2^j]}}_{L^1(\mathbb{R}^n)}^\frac{1}{n}
      \dotsm
      \norm[L^1(\mathbb{R}^n)]{\dualprod{v_n^j}{D u[w_n^j]}}^\frac{1}{n}
  .    
\] 
Now the construction of \cref{lemma_structure_n_1_canceling} with $v=\widehat{u} (\xi)$, the Fourier transform of $u$ at the point \(\xi \in \Rset^n\), yields for every \(j \in \{1, \dotsc, m\}\) and \(i \in \{1, \dotsc, n\}\) and \(\xi \in \Rset^n\),
\[
\dualprod{v_i^j}{\dualprod{2\pi i\xi}{w_i^j} \widehat{u} (\xi)}
      =
        \dualprod{e_i^j}{A (2\pi i \xi)\widehat{u} (\xi)},
\]
and thus inverting the Fourier transform we obtain the pointwise equality
\[
\dualprod{v_i^j}{Du[w_i^j]} = \dualprod{v_i^j}{\dualprod{Du}{w_i^j}} = 
        \dualprod{e_i^j}{A (D)u}.
\]
Hence
\[
 \norm[L^1 (\Rset^n)]{\dualprod{v_i^j}{D u[w_i^j]}}
 = \norm[L^1 (\Rset^n)]{\dualprod{e_i^j}{A (D) u}},
\]
and the conclusion then follows.
\end{proof}

\begin{proof}[Proof of \cref{theorem_canceling_2d}]
This follows from \cref{theorem_n_1_canceling}, \cite{VanSchaftingen_2013}*{Theorem 1.3}, the embedding between Lorentz spaces \(L^{n/(n -1), 1} (\Rset^n) \subset L^{n/(n -1)} (\Rset^n)\) and the fact that the \(1\)--canceling and canceling conditions are equivalent.
\end{proof}

\section*{Acknowledgements}
The authors would like to thank Wen-Wei Lin, the S.T. Yau Center at National Chiao Tung University, and the National Center for Theoretical Sciences of Taiwan for their support in the conference where this collaboration was initiated.  
The first author is supported in part by the Taiwan Ministry of Science and Technology under research grants 105-2115-M-009-004-MY2, 107-2918-I-009-003 and 
107-2115-M-009-002-MY2.

\end{document}